%% file: cpr-final-final.tex
\documentclass[graybox]{svmult}

\usepackage{amssymb,amsfonts,amsmath}
\usepackage[super]{nth}
\usepackage{mathptmx}       
\usepackage{helvet}         
\usepackage{courier}        
\usepackage{type1cm}        
\usepackage{makeidx}         
\usepackage{graphicx}        
\usepackage{multicol}        
\usepackage[bottom]{footmisc}

\makeindex             

\begin{document}

\title*{An overview of the balanced excited random walk}

\author{Daniel Camarena, Gonzalo Panizo and Alejandro 
  F. Ram\'\i rez}

\institute{Daniel Camarena \at Instituto de Matem\'atica y Ciencias 
  Afines,  Universidad Nacional de Ingenier\'\i a 
  Calle Los Bi\'ologos 245- Urb. San C\'esar-Primera Etapa 
  La Molina, Lima 12, Per\'u, \email{vcamarenap@uni.pe}
\and Gonzalo Panizo \at Instituto de Matem\'atica y Ciencias 
  Afines,  Universidad Nacional de Ingenier\'\i a 
  Calle Los Bi\'ologos 245- Urb. San C\'esar-Primera Etapa 
  La Molina, Lima 12, Per\'u, 
\email{  gonzalo.panizo@gmail.com}
\and Alejandro F. Ram\'\i rez \at
Facultad de Matem\'aticas
  Pontificia Universidad Cat\'olica de Chile
  Vicu\~na Mackenna 4860, Macul
  Santiago, Chile
\email{aramirez@mat.uc.cl}}

%
%
\maketitle

\abstract{ The balanced excited random walk, introduced by 
  Benjamini, Kozma and Schapira in $2011$, is defined as a discrete 
  time stochastic process in $\mathbb Z^d$, depending on two integer 
  parameters $1\le d_1,d_2\le d$, 
  which whenever it is at a site $x\in\mathbb Z^d$ at time $n$, it 
  jumps to $x\pm e_i$ with uniform probability, where $e_1,\ldots,e_d$ are the canonical 
  vectors, for $1\le i\le d_1$, if the site $x$ was visited for the 
  first time at time $n$, while it jumps to $x\pm e_i$ with uniform probability, for 
  $1+d-d_2\le i\le d$, if the site $x$ was already visited before time 
  $n$. Here we give an overview of this model when $d_1+d_2=d$ and 
  introduce and study the cases when $d_1+d_2>d$. In particular, we 
  prove that for all the cases $d\ge 5$ and most cases $d=4$, the 
  balanced excited random walk is transient.}

\section{Introduction}
\label{introduction}

We consider an extended version of the balanced 
excited random walk introduced by Benjamini, Kozma and Schapira in 
\cite{BKS11}. The balanced excited random walk is defined in any 
dimension $d\ge 2$, and depends on two integers 
$d_1,d_2\in \{1,\ldots,d\}$. For each $1\le i\le d$, let 
$e_i=(0,\ldots,0,1,0,\ldots,0)$ be the canonical vector whose $i$-th 
coordinate is $1$, while all other coordinates are $0$. We define the 
process $(S_n:n\ge 0)$, called the {\it balanced excited random walk}
on $\mathbb Z^d$ as a mixture of two simple random walks, with the 
initial condition $S_0=0$: if at time $n$, $S_n$ visits a site for the 
first time, with probability $1/(2 d_1)$, at time $n+1$ it performs a 
simple random walk step using one of the first $d_1$ coordinates, so 
that for all $1\le i\le d_1$, 

$$
\mathbb P[S_{n+1}-S_n=e_i|\mathcal F_n, S_n\ne S_j\ {\rm for}\ {\rm 
  all}\ 1\le j<n]
=\frac{1}{2 d_1}, 
$$
where $\mathcal F_n$ is the $\sigma$-algebra generated by 
$S_0,\ldots, S_n$; on the other hand, if at time $n$, $S_n$ visits a 
site it has previously visited, at time $n+1$ it performs a simple 
random walk using one the last $d_2$ coordinates, so that for all 
$d-d_2+1\le i\le d$, 

$$
\mathbb P[S_{n+1}-S_n=e_i|\mathcal F_n, S_n= S_j\ {\rm for}\ {\rm 
  some}\ 1\le j<n]=\frac{1}{2 d_2}. 
$$

We call this process $S$ the $M_d(d_1,d_2)$-random walk. In 
\cite{BKS11}, this random walk was considered in the case when 
$d_1+d_2=d$, which we call the {\it non-overlapping} case. Here we 
will focus on the {\it overlapping} case corresponding to $d_1+d_2>d$. 

We say that the $M_d(d_1,d_2)$-random walk is transient if any 
site is visited only finitely many times, while we say that it 
is recurrent if it visits every site infinitely often. 
Since a random walk $M_d(d_1,d_2)$ is not Markovian, in principle could be neither transient nor recurrent. 

For the non-overlapping case, in 2011 in \cite{BKS11} it was shown 
that the  $M_4(2,2)$-random walk is transient, while in 2016, Peres, 
Schapira and Sousi in \cite{PSS16}, showed that the 
$M_3(1,2)$-random walk is transient, but the transience of  $M_3(2,1)$-random walk is still an open question. 

The main result of this article is the following theorem concerned 
with the overlapping case. 

\medskip 

\begin{theorem}
	\label{proposition} For every $(d,d_1,d_2)$ with $d\ge 4$, $1\le d_1,d_2\le d$, 
         $d_1+d_2>d$ and $(d,d_1,d_2)\ne (4, 3,2)$, the $M_d(d_1,d_2)$-random walk is transient. 
\end{theorem}



\medskip 


Theorem \ref{proposition} has a simple proof for $d\ge 7$, for all
admissible values of $d_1$ and $d_2$. 
 Let $r:=d_1+d_2-d$. Note that if $r\ge 3$ then the walk is transient,
 since its
 restriction to the $r$ overlapping coordinates is at least a
 $3$-dimensional simple symmetric random walk with geometrically
 bounded holding times. We will argue in the next two paragraphs that
 the walk is also transient if $d_1-r\ge 3$ or if $d_2-r\ge 3$.
 Assuming for the moment that each of the three inequalities
 $r\ge 3$, $d_1-r\ge 3$ or $d_2-r\ge 3$ implies transience, note
 that if none of them holds we have that $d=d_1+d_2-r\le 6$. We
 conclude
 that for $d\ge 7$ the walk is transient for all admissible values of
 $d_1$ and $d_2$.

 {\it Case $d_2-r\ge 3$.} We claim that with probability $1$ the
 fraction of times when the random walk uses the last $d_2-r$
 coordinates
 is asymptotically bounded from below by a positive constant
 and therefore, the random walk is transient. To see this note that
 whenever
 the walk makes $3$ consecutive steps, the probability that in at
 least
 one of these steps it visits a previously visited (old) site is
 bounded away from $0$. Indeed, if in two consecutive steps
 the walk visits two previously unvisited (new) sites then with
 probability $1/(2d_1)$ it backtracks in the next step and, thus,
 visits an old site.

 {\it Case $d_1-r\ge 3$.}  We claim that the number of times the
 random walk uses the first $d_1-r$
 coordinates goes to infinity as $n\to\infty$, which is enough
 to prove transience.
 Denote by $r_n$ the number of points in the range of the walk
 at time $n$. We will show that $r_n\to\infty$ as $n\to\infty$ a.s.
For $k\ge 1$, let $n_k=\inf\{n\ge 0: r_n=k\}$.
 We will argue that if 
   $n_k<  \infty$ then with  probability one 
 $n_{k+1}<\infty$. Note that $r_{n_k}=k$, $S_{n_k}$ is a new site, and
 there
 are $k-1$ other sites in the range. Let $n_k<\infty$ and $A_1$ be the
 event that in the next $k$ steps the walk
 jumps only in positive coordinate directions. On $A_1$, at times
 $n_k+1, n_k+2,\ldots, n_k+k$ the walk visits $k$
 distinct sites of $\mathbb Z^d-\{S_{n_k}\}$. Among these sites there
 are at most $k-1$ old sites. Therefore, on the event
 $A_1\cap\{n_k<\infty\}$ the walk will necessarily visit a new site
 and
 $n_{k+1}\le n_k+k<\infty$. Note that the probability of $A_1$ (given
 $n_k<\infty$)
 is $2^{-k}$. If $A_1$ does not occur, then we consider the next $k$
 steps
 and define $A_2$ to be the event that in these next $k$ steps the
 walk
 jumps only in the positive coordinate directions, and so on. Since,
 conditional on $n_k<\infty$, the events $A_1,A_2,\ldots$ are
 independent
 and each has probability $2^{-k}$, we conclude that $n_{k+1}<\infty$
 with probability one.

 Therefore, to complete the proof of Theorem \ref{proposition} we have
 only
 to consider the cases $d=4,5,6$. It will be shown below that the cases 
$d=5,6$ and several cases in $d=4$, can be derived in an elementary way 
sometimes using the trace condition of \cite{PPS13}. In a less straightforward way the 
cases $M_4(2,4)$ and $M_4(4,2)$ can be treated through the methods of \cite{BKS11}. 
The case $M_4(2,3)$ which is more involved, can be treated through a modification 
of methods developed by Peres, Schapira and Sousi \cite{PSS16} for the 
$M_3(1,2)$-random walk through good controls on martingale increments by sequences of geometric 
i.i.d. random variables. It is not clear how the above 
mentioned 
methods could be applied to the  $M_4(3,2)$-random walk to settle down the 
transience-recurrence question for it, so this case remains open. 


In Section \ref{section-nonoverlapping} we will give a quick review of 
the main results that have been previously obtained for the 
non-overlapping case of the balanced excited random walk. In Section 
\ref{prooft}, we will prove Theorem \ref{proposition}.  In Section 
\ref{d=54}, we will introduce the trace condition of \cite{PPS13}, 
which will be used to prove the cases $d=5,6$ and several cases 
in dimension $d=4$. In Section \ref{s24}, we will prove the transience 
of the random walks $M_4(2,4)$ and $M_4(4,2)$. While in Section 
\ref{s32}, we will consider the proof of the transience of the 
$M_4(2,3)$-random walk.

\section{Overview of the balanced excited random walk}
\label{section-nonoverlapping}

The balanced excited random walk was introduced in its non-overlapping 
version by Benjamini, Kozma and Schapira in \cite{BKS11}. A precursor 
of the balanced excited random walk, is the {\it excited random walk}, 
introduced by Benjamini and Wilson in 2003 \cite{BW03}, which is 
defined in terms of a parameter $0<p<1$ as follows: the random walk 
$(X_n:n\ge 0)$ has the state space $\mathbb Z^d$ starting at $X_0=0$; 
whenever the random walk visits a site for the first time, it jumps 
with probability $(1+p)/2d$ in direction $e_1$, probability $(1-p)/2d$
in direction $-e_1$ and with probability $1/2d$ in the other 
directions; whenever the random walk visits a site which it already 
visited previously it jumps with uniform probability in directions 
$\pm e_i, 1\le i\le d$.  Benjamini and Wilson proved in \cite{BW03}
that the model is transient for $d>1$. A central limit theorem and a 
law of large numbers for $d>1$ was proven in \cite{BR07} and 
\cite{K07}.  A general review of the model can be found in 
\cite{KZ13}. Often the methods used to prove transience, 
the law of large numbers and the central limit theorem for the excited 
random walk, are based on the ballisticity of the model (the fact that the 
velocity is non-zero), through the use of regeneration times. This 
means that most of these methods are not well suited to study the 
balanced excited random walk, which is not ballistic. For the moment, 
a few results have been obtained for the balanced excited random walk, 
where basically for each case a different technique has been 
developed.  The first result was obtained by Benjamini, Kozma and 
Schapira in \cite{BKS11} for the $M_4(2,2)$ case is the following 
theorem. 

\medskip 

\begin{theorem}[{\bf Benjamini, Kozma and Schapira, 2011}]
  \label{bks} The $M_4(2,2)$-random walk is transient. 
\end{theorem}

\medskip 

The proof of Theorem \ref{bks} is based on obtaining good enough 
estimates for the probability that a $2$-dimensional random walk 
returns to its starting point in a time interval $[n/c(\log n)^2, cn]$, 
for some constant $c>0$, and on the range of the random walk. 
This then allows to decouple using independence the first $2$
coordinates from the last $2$ ones. In this article, we will apply 
this method to derive the transience in the $M_4(4,2)$ and $M_4(2,4)$
cases of Theorem \ref{proposition}. 

In 2016, Peres, Sousi and Schapira in \cite{PSS16}, considered the 
case $M_3(1,2)$ proving the following result. 

\medskip 

\begin{theorem}[{\bf Peres, Schapira and Sousi, 2016}]
\label{pps}
The $M_3(1,2)$-random walk is transient. 
\end{theorem}

\medskip 
  
The approach developed in \cite{PSS16} to prove Theorem \ref{pps}, 
starts with conditioning on all the jumps of the last two coordinates, 
and then looking at the first coordinate at the times when the last 
two move, which gives a martingale. It is then enough to obtain good 
estimates on the probability that this martingale is at $0$ at time 
$n$.  The proof of the $M_4(2,3)$-random walk case of Theorem 
\ref{proposition}, is based on a modification of the method used to 
prove Theorem \ref{pps}, where a key point is to obtain appropriate 
bounds for martingale increments (which will correspond to the first 
coordinate of the movement of the $M_4(2,3)$-random walk) in terms of 
i.i.d. sequences of geometric random variables.

\section{Proof of Theorem \ref{proposition}}
\label{prooft}

We will divide the proof of Theorem 
\ref{proposition} in three steps.  With the exception of the cases 
$M_4(1,4)$, $M_4(4,1)$, $M_4(2,4)$, $M_4(4,2)$ and $M_4(2,3)$, we will use an important result 
of Peres, Popov and Sousi \cite{PPS13}. The cases $M_4(1,4)$ and $M_4(4,1)$ will be derived as 
those in dimension $d\ge 7$. For the cases $M_4(2,4)$ and 
$M_4(4,2)$ we will show how the argument of \cite{BKS11} can be 
adapted. And the case $ M_4(2,3)$ is handled as in \cite{PSS16}. 

\medskip 

\subsection{The trace condition}
\label{d=54}
Here we will recall the so called trace condition of \cite{PPS13}
which is a general condition under which a generalized version of the 
balanced random walk is transient, 
and see how it can be used 
to prove Theorem \ref{proposition} for the cases different from 
$M_4(3,4)$, $M_4(4,3), M_4(3,3)$ and $M_4(4,4)$. 

Given $d\ge 1$ and $m\ge 1$, consider probability measures 
$\mu_1,\ldots,\mu_m$ on $\mathbb R^d$ and for each 
$1\le i\le m$, let $(\xi_n^i:n\ge 1)$ be an i.i.d. sequence 
of random variables distributed according to $\mu_i$. 
We say that a stochastic process $(\ell_k:k\ge 0)$ is 
an {\it adapted rule} with respect to a 
 filtration $(\mathcal{F}_n:n\ge 0)$ of the process, if for each $k\ge 
 0$, 
 $\ell_k$ is $\mathcal{F}_k$-measurable. 
We now define the random walk $(X_n:n\ge 0)$ generated 
by the probability measures $\mu_1,\ldots,\mu_m$ and the 
adapted rule $\ell$ by 

$$
X_{n+1}=X_n+\xi^{\ell_n}_{n+1},\qquad {\rm for}\quad n\ge 0. 
$$
Let   $\mu$ be a measure on $\mathbb R^d$. $\mu$ is called of {\it mean $0$} if $\int 
xd\mu=0$. 
The measure $\mu$ is said to have {\it $\beta$ moments} if for any 
random variable $Z$ distributed according to $\mu$, $||Z||$ has moment 
of order $\beta$. The covariance matrix of $\mu$, $Var(\mu)$, is defined as the 
covariance 
of $Z$. 

Given a matrix $A$, we call $\lambda_{max}(A)$ its maximal eigenvalue 
and $A^t$
its transpose. In \cite{PPS13}, the following result was proven. 
\medskip 


\begin{theorem}[Peres, Popov and Sousi, 2013]
  \label{theorempopov}
  Let $\mu_1,\ldots,\mu_m$ be measures in $\mathbb R^d$, $d\ge 3$, with 
  zero mean and $2+\beta$ moments, for some $\beta>0$.  Assume that 
  there is a matrix $A$ such that the trace condition is satisfied: 

  $$
  tr(A \ Var(\mu_i)\ A^t)>2\lambda_{max}(A\ Var(\mu_i)\ A^t) 
$$
for all $1\le i\le m$. Then any random walk $X$ generated by these 
measures and any adapted rule is transient. 
\end{theorem}

\medskip 


It follows from Theorem \ref{theorempopov}, that whenever $d_1\ge 3$ and $d_2\ge 3$, the trace condition is satisfied, 
with $A=I$, 
for the two corresponding matrices associated to the first $d_1$ and last $d_2$ dimensions, and 
hence the $M_d(d_1,d_2)$-random walk is transient. Hence, by the discussion right after the statement of Theorem \ref{proposition} in 
Section \ref{introduction}, we see that the only cases which are not 
covered by Theorem \ref{theorempopov}, correspond to 

\begin{equation}
  \label{ddr}
  d_1-r\le 2, ~r\le 2\quad {\rm and}\quad d_2-r\le 2, 
 \end{equation}
 and 

 $$
 \min\{d_1,d_2\}\le 2. 
 $$
 But (\ref{ddr}) implies that $\max\{d_1,d_2\}\le 2+r$. Thus, 

 $$
 d_1+d_2=\max\{d_1,d_2\}+\min\{d_1,d_2\}\le 4+r, 
 $$
 so that $d=d_1+d_2-r\le 4$. This proves the transience for all the 
 cases when $d\ge 5$. Now note that in dimesion $d=4$ the random walks $M_4(3,3), M_4(3,4), M_4(4,3)$ and $M_4(4,4)$ satisfy $d_1\ge 3$ and $d_2\ge 3$, so that the trace condition of \cite{PPS13} is satisfied. 
 
 Finally, that the random walks $M_4(1,4)$ and $M_4(4,1)$ satisfy $r\ge 3$, $d_1 -r\ge 3$ or $d_2 -r\ge 3$, so that they are also transient. 

 \medskip

 \subsection{The random walks $M_4(2,4)$ and $M_4(4,2)$}
 \label{s24}
 Consider the $M_4(4,2)$-random walk and call $r_n$ the cardinality of its range at time 
 $n$.  Let us use the notation $S=(X,Y)$ for the $M_4(4,2)$-random 
 walk, where $X$ are the first two components and $Y$ the last two 
 ones. We will also call $r_n^{(1)}$ the number of times up to time 
 $n$ that the random walk jumped using the $X$ coordinates while it 
 was at a site that it visited for the first time and 
 $r_n^{(2)}:=r_n-r^{(1)}_n$.  In analogy with Lemma 1 of \cite{BKS11}, 
 we have the following result. 

 \medskip 

 \begin{lemma}
   \label{lemma1} For any $M>0$ and each $i=1,2$, there exists a 
   constant $C>0$ such that 

   \begin{equation}
     \label{ri}\mathbb P[n/(C\log n)^2\le r^{(i)}_n\le 99n/100]=1-o\left(n^{-M}\right). 
   \end{equation}
 \end{lemma}
 \begin{proof} \noindent First note that in analogy to the proof Lemma 
   1 of \cite{BKS11}, we have that 

$$
\mathbb P[n/(C\log n)^2\le r_n\le 99n/100]=1-o\left(n^{-M}\right). 
$$
Since each time the random walk is at a newly visited site with 
probability $1/2$ it jumps using the $X$ random walk and with 
probability $1/2$ the $Y$ random walk, by standard large deviation 
estimates, we deduce (\ref{ri}). 
\end{proof}
\medskip

Now note that 
\begin{equation}
  \label{xy}
  \{(X_k,Y_k): k\ge 1\}=\{(U_1(r^{(1)}_{k-1}),~U_2(r^{(2)}_{k-1})+V(k-r_{k-1})):k\ge 1\}, 
\end{equation}
where $U_1$, $U_2$ and $V$ are three independent simple random walks 
in $\mathbb Z^2$. It follows from the identity (\ref{xy}) and Lemma 
\ref{lemma1} used to bound the components $r^{(1)}_n$ and $r^{(2)}_n$
of the range of the walk, that 

\begin{eqnarray}
  \nonumber 
  &
    \mathbb P[0\in \{S_n,\ldots,S_{2n}\}]\le\mathbb P[0\in \{ U(n/(C\log 
    n)^2 ),\ldots, U(2n)]\\
  \label{os}
  &\times\mathbb P[0\in \{ W(n/(C\log 
    n)^2 ),\ldots, W(2n)]+o(n^{-M}), 
\end{eqnarray}
where $U$ and $W$ are simple symmetric random walks on $\mathbb Z^2$. 
At this point, we recall Lemma 2 of \cite{BKS11}. 

\medskip 

\begin{lemma}[Benjamini, Kozma and Schapira, 2011]%
  \label{lemma2} Let $U$ be a simple random walk on $\mathbb Z^2$ and 
  let $t\in [n/(\log n)^3,2n]$. Then 

$$
\mathbb P[0\in \{ U(t ),\ldots, U(2n)]=O\left(\frac{\log\log n}{\log 
    n}\right). 
$$
\end{lemma}

\medskip 

Combining inequality (\ref{os}) with Lemma \ref{lemma2}, we conclude 
that there is a constant $C>0$ such that for any $n>1$ (see 
Proposition 1 of \cite{BKS11}) 
$$
\mathbb P[0\in \{ S_n ,\ldots, S_{2n}\}]\le C\left(\frac{\log\log 
    n}{\log n}\right)^2. 
$$
Hence, 

$$
\sum_{k=0}^\infty \mathbb P[0\in \{ S_{2^k} ,\ldots, 
S_{2^{k+1}}\}]<\infty, 
$$
and the transience of the $M_4(4,2)$-random walk follows form 
Borel-Cantelli. A similar argument can be used to prove the transience 
of the $M_4(2,4)$-random walk. 

\medskip 

\subsection{The $M_4(2,3)$-random walk}
\label{s32}
Here we will follow the method developed by Peres, Schapira and Sousi 
in \cite{PSS16}. We first state Proposition 2.1 of \cite{PSS16}.

\medskip 

\begin{proposition}[Peres, Schapira and Sousi, 2016]
  \label{propo}
  Let $\rho>0$ and $C_1, C_2>0$. Let $M$ be a martingale with 
  quadratic variation $V$ and assume that $(G_k:k\ge 0)$ is a sequence 
  of $i.i.d.$ geometric random variables with mean $C_1$ such that for 
  all $k\ge 0$, 

  \begin{equation}
    \label{six}
    |M_{k+1}-M_k|\le C_2 G_k. 
  \end{equation}
  For all $n\ge 1$ and $1\le k\le\log_2(n)$ let $t_k:=n-\frac{n}{2^k}$
  and 

  $$
  A_k:=\left\{V_{t_{k+1}}-V_{t_k}\ge\rho\frac{t_{k+1}-t_k}{(\log 
      n)^{2a}}\right\}. 
  $$
  Suppose that for some $N\ge 1$ and 
  $1\le k_1<\cdots< k_N<\log_2(n)/2$ one has that 

  \begin{equation}
    \label{five}
    \mathbb P\left(\cap_{i=1}^N A_{k_i}\right)=1. 
  \end{equation}
  Then, there exists constant $c>0$ and a positive integer $n_0$ such 
  that for all $a\in (0,1)$ and $n\ge n_0$ one has that 

  $$
  \mathbb P(M_n=0)\le\exp\left(-cN/(\log n)^a\right). 
  $$

\end{proposition}

\medskip 

\begin{remark}
  Proposition \ref{propo} is slightly modified with respect to 
  Proposition 2.1 of \cite{PSS16} since we have allowed the mean $C_1$
  of the geometric random variables to be arbitrary and the bound 
  (\ref{six}) to have an arbitrary constant $C_2$. 
\end{remark}

\medskip 


Let us now note that the $M_4(2,3)$-random walk $(S_n:n\ge 0)$ can be 
defined as follows. Suppose $(\zeta_n:n\ge 1)$ is a sequence of 
i.i.d. random variables taking each of the values $(0,\pm 1, 0, 0)$, 
$(0, 0, \pm 1, 0)$ and $(0, 0, 0, \pm 1)$ with probability $1/6$, 
while $(\xi_n:n\ge 1)$ is a sequence of i.i.d. random variables 
(independent from the previous sequence) taking each of the values 
$(0,\pm 1, 0, 0)$ and $(\pm 1, 0, 0, 0)$ with probability 
$1/4$. Define now recursively, $S_0=0$, and 
$$
S_{n+1}=S_n+\Delta_{n+1}
$$
where the step is 

$$
\Delta_{n+1}= \begin{cases} \xi_{r_n}\,, & \text{if }r_n=r_{n-1}+1\\
  \zeta_{n+1-r_n}\,, & \text{if }r_n=r_{n-1} 
  \end{cases}
$$
and $r_n=\#\{S_0,\dots,S_n\}$ as before is the cardinality of the range of the random walk at time 
$n$ (note that formally $r_{-1}=0$).

Let us now write the position at time $n$ of the $M_4(2,3)$ random 
walk as 

$$
S_n=(X_n,Y_n,Z_n,W_n). 
$$
Define recursively the sequence of stopping times 
$(\tau_k:k\ge 0)$ by $\tau_0=0$ and for $k\ge 1$, 

$$
\tau_k:=\inf\{n>\tau_{k-1}: (Z_n,W_n)\ne (Z_{n-1},W_{n-1})\}. 
$$
Note that $r_0=1$ and $\tau_k<\infty$ a.s. for all $k\ge 
0$. Furthermore, the process $(U_k:k\ge 0)$ defined by 

$$
U_k=(Z_{\tau_k}, W_{\tau_k}), 
$$
is a simple random walk in dimension $d=2$, and is equal to the simple 
random walk with steps defined by the last two coordinates of $\zeta$. Let us now 
call $P_U$ the law of $S$ conditionally on the whole $U$ process. Note 
that the first coordinate $\{X_n:n\ge 0\}$ is an 
$\mathcal F_n:=\sigma\{\Delta_k:k\le n\}$-martingale with respect to 
$P_U$, since 

$$
E_U(\,X_{n+1}-X_n\;|\;\mathcal 
F_n)=1_{r_n=r_{n-1}+1}\,E(\,\xi_{r_n}\cdot e_1\,|\,\mathcal F_n,\,U), 
$$
$U$ is $\sigma(\zeta_k:k\ge 1)$-measurable as it is defined only in 
terms of the sequence 
$\big(\zeta_k1_{\{\pi_{34}(\zeta_k)\neq 0\}}\big)_{k\ge 
  1}$, ($\pi_{34}$ being the projection in the \nth{3} and \nth{4}
coordinates), and 

$$
E[\,\xi_{r_n}\cdot e_1\,|\,\mathcal F_n,(\zeta_k:k\ge 1)]=0, 
$$
by independence.  Hence, $\{M_m:m\ge 0\}$ with $M_m:=X_{\tau_m}$, is a 
$\mathcal G_m$-martingale with respect to $P_U$, where 
$\mathcal G_m:=\mathcal F_{\tau_m}$.  To prove the theorem, it is 
enough to show that $\{(M_n,U_n):n\ge 0\}$ is transient (under $P$). 
Let us call $r_U(n)$ the cardinality of 
the range of the random walk 
$U$ at time $n$. For each $n\ge 0$ and $k\ge 0$, let 
\begin{equation}
  \label{tk}
  t_k:=n-n/2^k 
\end{equation}
and 
\begin{equation}
  \label{kk}
  \mathcal K:=\left\{k\in\{1,\ldots, (\log 
    n)^{3/4}\}:r_U(t_{k+1})-r_U(t_k)\ge\rho 
    (t_{k+1}-t_k)/\log n\right\}. 
\end{equation}
We will show that 

\begin{eqnarray}
  \nonumber 
  &P(M_n=U_n=0) 
  =E[P_U(M_n=0)1\{|\mathcal K|\ge\rho(\log n)^{3/4}, 
    U_n=0\}]\\   \label{four}
   & +E[P_U(M_n=0)1\{|\mathcal K|<\rho(\log n)^{3/4}, U_n=0\}], 
\end{eqnarray}
is summable in $n$, for $\rho=\rho_0$ chosen appropriately.  At this point, let 
us recall Proposition 3.4 of \cite{PSS16}, which is a statement about 
simple symmetric random walks. 

\medskip

\begin{proposition}[Peres, Schapira and Sousi, 2016] \label{prop2}
  For $k\ge 1$, consider $t_k$ as defined in (\ref{tk}). Then, for 
  $\mathcal K$ as defined in (\ref{kk}), we have that there exist 
  positive constants $\alpha, C_3, C_4$ and $\rho_*$, such that for 
  all $\rho<\rho_*$, 

$$
P(|\mathcal K|\le\rho (\log n)^{3/4}|U_n=0)\le C_3 e^{-C_4(\log 
  n)^\alpha}. 
$$
\end{proposition}

\medskip

Choosing $\rho=\rho_0\le 1$ small enough, by Proposition \ref{prop2}, we have the 
following bound for the second term on the right-hand side of (\ref{four}), 

\begin{equation}
  \label{last}
  E[P_U(M_n=0)1\{|\mathcal K|<\rho_0(\log n)^{3/4}, U_n=0\}]
  \le C_3C_5\frac{1}{n}\exp\left(-C_6(\log n)^\alpha\right), 
\end{equation}
for some positive $C_6$ and $\alpha$, where we have used the fact that 
$P(U_n=0)\le \frac{C_5}{n}$ for some constant $C_5>0$. 

To bound the first term on the right-hand side of (\ref{four}), we 
will use Proposition \ref{propo} with $a=1/2$ and $\rho=\rho_0/4$. 
Let us first show that (\ref{five}) is satisfied. Indeed, note that 
for each $n\ge 0$ when $U_n$ is at a new site, 
$E_U[(M_{n+1}-M_n)^2|\mathcal G_n]\ge 1/2$. Therefore, for all 
$k\in\mathcal K$, with $\rho=\rho_0$, one has that 


\begin{eqnarray*}
  &V_{t_{k+1}}-V_{t_k}=\sum_{n=t_k+1}^{t_{k+1}}
    E_U[(M_n-M_{n-1})^2|\mathcal G_{n-1}]\\
  &\ge \big(r_U(t_{k+1}-1)-r_U(t_k-1)\big)/2 
    \ge \big(r_U(t_{k+1})-r_U(t_k)-1\big)/2 \\
  &\ge (\rho_0/4)(t_{k+1}-t_k)/(\log n)^{2a}. 
\end{eqnarray*}
Hence, on the event $|\mathcal K|\ge \rho_0(\log n)^{3/4}$, we have 
that there exist $k_1,\ldots,k_N\in\mathcal K$ with 
$N=[\rho_0(\log n)^{3/4}]$ such that 

$$
P_U\left(\cap_{i=1}^NA_{k_i}\right)=1. 
$$
Let us now show that there is a sequence of i.i.d. random variables 
$(G_k:k\ge 0)$ such that (\ref{six}) is satisfied with $C_1=24$ and 
$C_2=3$. Indeed, note that 

\begin{equation}
  \label{steps}
  |M_{n+1}-M_n|=|X_{\tau_{n+1}}-X_{\tau_n}|\le\sum_{k=\tau_n}^\infty 
  |X_{(k+1)\land\tau_{n+1}}-X_{k\land\tau_{n+1}}|. 
\end{equation}
Note that the right-hand side of (\ref{steps}) is the number of steps 
of $X$ between times $\tau_n$ and $\tau_{n+1}$.  Now, at each time $k$
(with $k$ starting at $\tau_n$) that a step in $X$ is made there is a 
probability of at least $\frac{1}{4^2}\times\frac{2}{3}=\frac{1}{24}$
that the random walk $S$ makes three succesive steps at times $k+1$, 
$k+2$ and $k+3$, in such a way that in one of them a step in $U$ is 
made and at most two of these steps are of the $X$ random walk: if the 
random walk is at a site previously visited at time $k$, with 
probability $2/3$ at time $k+1$ the $U$ random walk will move; if the 
random walk is at a site which it had never visited before at time 
$k$, with probability $\frac{1}{4^2}\times\frac{2}{3}=\frac{1}{24}$, 
there will be $3$ succesive steps of $S$ at times $k+1$, $k+2$ and 
$k+3$, with the first $2$ steps being of the $X$ random walk and the 
third step of $U$ (we just need to move in the $e_1$ direction using 
$X$ at time $k+1$, immediately follow it at time $k+2$ by a reverse 
step in the $-e_1$ direction using $X$ again, and then immediately at 
time $k+3$ do a step in $U$). Since this happens independently each 
$3$ steps in the time scale of $X$ (time increases by one unit 
whenever $X$ moves), we see that we can bound the martingale 
increments choosing i.i.d. geometric random variables $(G_k:k\ge 0)$
of parameter $1/24$ in (\ref{six}) multiplied by $3$. 


\medskip 

\begin{remark} The sequence of i.i.d. geometric random variables 
  constructed above is not the optimal one, in the sense that it is 
  possible to construct other sequences of i.i.d. geometric random 
  variables of parameter larger than $1/24$. 
\end{remark}
\medskip

Since now we know that (\ref{five}) and (\ref{six}) are satisfied, by 
Proposition \ref{propo}, there exist $n_0\ge 1$ and $C_7>0$ such that 
on the event $|\mathcal K|\ge \rho_0(\log n)^{3/4}$ we have that for $n\ge n_0$, 

$$
P_U(M_n=0)\le e^{-C_7\rho_0\frac{(\log n)^{3/4}}{(\log n)^{1/2}}}. 
$$
Hence, for $n\ge n_0$ we have 

\begin{equation}
  \label{last2}
  E[P_U(M_n=0)1\{|\mathcal K|\ge\rho_0(\log n)^{3/4}, 
  U_n=0\}]\le C_5\frac{1}{n}e^{-C_7\rho_0\frac{(\log n)^{3/4}}{(\log n)^{1/2}}}. 
\end{equation}
Using the bounds (\ref{last}) and (\ref{last2}) back in (\ref{four}) 
gives us that there exist constants $C_8>0$, $C_9>0$ and some 
$\beta>0$, such that

$$
P(M_n=U_n=0)\le\frac{1}{n} C_8e^{-C_9(\log n)^\beta}
$$
By the Borel-Cantelli lemma, we conclude that the process $(M,U)$ is 
transient, which gives the transience of $S$.

\begin{acknowledgement}
Daniel Camarena and Gonzalo Panizo thank the support of Fondo Nacional de Desarrollo 
  Cient\'\i fico, Tecnol\'ogico y de Innovaci\'on Tecnol\'ogica 
  CG-176-2015. Alejandro Ram\'\i rez thanks the support of
 Iniciativa Cient\'\i fica Milenio 
  and of Fondo Nacional de Desarrollo Cient\'\i fico y Tecnol\'ogico 
  grant 1180259
\end{acknowledgement}
%

\input{referenc}

\end{document}

%% file: referenc.tex
%
%
%